\documentclass[11pt]{article}
\usepackage{amsmath}
\usepackage{amssymb}
\usepackage{url}
\usepackage{graphicx}
\usepackage{ifpdf}
\def\RR{\boldsymbol{R}}

\def\BB {\boldsymbol{B}}
\def\CC {\boldsymbol{C}}

\def\FF{\boldsymbol{F}}

\def\qq{\boldsymbol{q}}

\def\rr{\boldsymbol{r}}

\def\BB {\boldsymbol{B}}

\def\LL {\boldsymbol{L}}
\def\AAA {\boldsymbol{A}}

\def\C{\mathbb{C}}

\def\xx{\boldsymbol{x}}

\def\dd{\boldsymbol{d}}

\def\ff{\boldsymbol{f}}
\def\hh{\boldsymbol{h}}
\def\ll{\boldsymbol{l}}
\def\pp{\boldsymbol{p}}
\def\yy{\boldsymbol{y}}
\def\bb{\boldsymbol{b}}

\def\uu{\boldsymbol{u}}
\def\vv{\boldsymbol{v}}

\def\cc{\boldsymbol{c}}
\def\sss{\boldsymbol{s}}

\def\UU{\boldsymbol{U}}

\def\VV{\boldsymbol{V}}
\def\GG{\boldsymbol{G}}
\def\HH{\boldsymbol{H}}

\def\pp{\boldsymbol{p}}
\def\QQ{\boldsymbol{Q}}
\def\RR{\boldsymbol{R}}
\def\TT{\boldsymbol{T}}

\def\YY{\boldsymbol{Y}}

\def\MM{\boldsymbol{M}}

\def\aaa{\boldsymbol{a}}
\def\gg{\boldsymbol{g}}
\def\hh{\boldsymbol{h}}
\def\II{\boldsymbol{I}}

\def\00{\boldsymbol{0}}
\def\11{\boldsymbol{1}}
\def\ggamma{\mbox{\boldmath{$\gamma$}}}

\def\eepsilon{\mbox{\boldmath{$\epsilon$}}}

\def\xxi{\mbox{\boldmath{$\xi$}}}
\def\ppsi{\mbox{\boldmath{$\psi$}}}
\def\eeta{\mbox{\boldmath{$\eta$}}}

\def\SSigma{\mbox{\boldmath{$\Sigma$}}}

\newcommand{\off}[1]{}

\newtheorem{theorem}{Theorem}
\newtheorem{lemma}[theorem]{Lemma}

\newtheorem{example}[theorem]{Example}
\newenvironment{proof}{\noindent{\bf{Proof.\/}}}{\hfill$\blacksquare$\vskip0.1in}

\newcommand{\beq}{\begin{equation}}
\newcommand{\eeq}{\end{equation}}

\newcommand{\ben}{\begin{enumerate}}
\newcommand{\een}{\end{enumerate}}

\newcommand{\rank}{\mbox{rank}}

\begin{document}
\title
{\bf  SVD-MPE: An SVD-Based Vector Extrapolation Method of Polynomial Type}
\author
{Avram Sidi\\
Computer Science Department\\
Technion - Israel Institute of Technology\\ Haifa 32000, Israel\\{~~}\\
e-mail:\ \ \url{asidi@cs.technion.ac.il}\\
URL:\ \ \url{http://www.cs.technion.ac.il/~asidi}}
\date{\today}
\maketitle

\thispagestyle{empty}
\newpage

\begin{abstract}\noindent
An important problem that arises in different areas of science and engineering is that of computing the limits of sequences of
vectors $\{\xx_m\}$, where $\xx_m\in \C^N$, $N$ being  very large. Such sequences arise, for example, in the solution of systems of linear or nonlinear
equations by fixed-point iterative methods, and $\lim_{m\to\infty}\xx_m$ are simply the required solutions. In most cases of interest, however,
these sequences converge to their limits extremely slowly. One practical way to make the sequences $\{\xx_m\}$ converge more quickly
is to apply to them vector extrapolation methods. Two types of methods exist in the literature: polynomial type methods and epsilon algorithms.
In most applications, the polynomial type methods
have proved to be superior  convergence accelerators.
Three polynomial type methods are known, and these are the {minimal polynomial extrapolation} (MPE),  the {reduced rank extrapolation} (RRE), and the
{modified minimal polynomial extrapolation} (MMPE). In this work, we develop yet another
polynomial type method, which is based on the singular value decomposition, as well as the ideas that lead to MPE. We denote this new  method by  SVD-MPE.  We also design a numerically stable  algorithm for its implementation, whose computational cost and storage requirements are minimal. Finally, we illustrate the use of {SVD-MPE} with
 numerical  examples.
\end{abstract}
\vspace{1cm} \noindent {\bf Mathematics Subject Classification
2000:} 15A18, 65B05, 65F10, 65F50, 65H10.

\vspace{1cm} \noindent {\bf Keywords and expressions:}
Vector extrapolation, minimal polynomial extrapolation, singular value decomposition, Krylov subspace methods.
\thispagestyle{empty}
\newpage
\pagenumbering{arabic}

\section{Introduction and background}\label{se1}
\setcounter{theorem}{0} \setcounter{equation}{0}
An important problem that arises in different areas of science and
engineering is that of computing limits of sequences of vectors
$\{{\xx}_m\}$,\footnote{Unless otherwise stated, $\{c_m\}$ will
mean $\{c_m\}_{m=0}^\infty$ throughout this work.} where
${\xx}_m\in\mathbb{C}^N$, the dimension $N$ being very large in
many applications. Such vector sequences arise, for example, in
the numerical solution of very large systems of linear or
nonlinear equations by fixed-point iterative methods, and
$\lim_{m\to\infty}{\xx}_m$ are simply the required solutions to
these systems. One common source of such systems is the
finite-difference or finite-element discretization of continuum
problems. In later chapters, we will discuss further problems that
give rise to vector sequences whose limits are needed.

In most cases of interest, however, the sequences $\{\xx_m\}$
converge to their limits extremely slowly. That is,  to
approximate $\sss=\lim_{m\to\infty}{\xx}_m$,  with a reasonable
prescribed level of accuracy, by $\xx_m$, we need to  consider
very large values of $m$. Since, the vectors $\xx_m$ are normally
computed in the order $m=0,1,2,\ldots,$ it is clear that we have to
compute many such vectors until we reach one that has acceptable
accuracy. Thus,  this way of approximating $\sss$ via the $\xx_m$
becomes very expensive computationally.

Nevertheless, we may ask whether we can do something with  those
$\xx_m$ that are already available, to somehow obtain new
approximations to $\sss$ that are better than each individual
available $\xx_m$. The answer to this question is in the
affirmative for at least a large class of sequences that arise
from fixed-point iteration of linear and nonlinear systems of
equations.  One practical way of achieving this is by applying to
the sequence $\{\xx_m\}$ a suitable {\em convergence acceleration
method} (or {\em extrapolation method}).

Of course, in case $\lim_{m\to\infty}{\xx}_m$ does not exist, it
seems that no use could be made of the $\xx_m$. Now, if the
sequence $\{\xx_m\}$ is generated by an iterative solution of a
linear or nonlinear system of equations, it can be thought of as
``diverging from'' the solution $\sss$ of this system. We call
$\sss$ the {\em antilimit} of $\{\xx_m\}$ in  such a case.  It
turns out that vector extrapolation methods can be applied to such
divergent sequences $\{\xx_m\}$ to obtain good approximations to
the relevant antilimits, at least in some  cases.

Two different types of vector extrapolation methods exist in the literature:
\begin{enumerate}
\item Polynomial type methods: The {\em minimal polynomial extrapolation} (MPE) of Cabay and Jackson \cite{Cabay:1976:PEM},
    the {\em reduced rank extrapolation} (RRE) of Eddy \cite{Eddy:1979:ELV} and Me{\u{s}}ina \cite{Mesina:1977:CAI}, and the {\em modified   minimal polynomial extrapolation} (MMPE) of Brezinski \cite{Brezinski:1975:GTS}, Pugachev \cite{Pugachev:1978:ACI}
    and Sidi, Ford, and Smith \cite{Sidi:1986:ACV}.
\item Epsilon algorithms: The {\em scalar epsilon algorithm} (SEA) of Wynn \cite{Wynn:1956:DCT} (which is actually a recursive procedure for implementing the transformation of Shanks \cite{Shanks:1955:NTD}), the  {\em vector epsilon algorithm} (VEA) of Wynn \cite{Wynn:1962:ATI}, and the {\em topological epsilon algorithm} (TEA) of Brezinski \cite{Brezinski:1975:GTS}.
\end{enumerate}

The paper by Smith, Ford, and Sidi \cite{Smith:1987:EMV} gives a review of all these methods (except MMPE) that covers the developments in vector extrapolation methods until the end of the 1970s. For  up-to-date reviews of MPE and RRE, see Sidi \cite{Sidi:2008:VEM} and \cite{Sidi:2012:RTV}.
Numerically stable  algorithms for implementing MPE and RRE are given in Sidi \cite{Sidi:1991:EIM},
these algorithms being also economical both computationally and storagewise.
For the convergence properties and error analyses of  MPE, RRE, MMPE, and TEA, as these are applied to
vector sequences generated by fixed-point iterative methods from linear systems,
see the works by Sidi \cite{Sidi:1986:CSP}, \cite{Sidi:1988:EVP}, \cite{Sidi:1994:CIR},
Sidi, Ford, and Smith \cite{Sidi:1986:ACV}, Sidi and Bridger \cite{Sidi:1988:CSA},
and Sidi and Shapira \cite{Sidi:1998:UBC}.  VEA has  been studied by Brezinski \cite{Brezinski:1970:AEA}, \cite{Brezinski:1971:ARS}, Gekeler \cite{Gekeler:1972:SSE},
Wynn \cite{Wynn:1963:CFW}, \cite{Wynn:1964:GPV}, and Graves--Morris \cite{Graves:1983:VVR-1},
\cite{Graves:1992:EMV}.

 Vector extrapolation methods are used effectively
in various branches of science and engineering
in accelerating the convergence of iterative methods that result from large sparse systems of equations.

All of these methods have the useful feature that their only input is the vector sequence $\{\xx_m\}$ whose convergence is to be accelerated, nothing else being needed.   In most applications, however, the polynomial type methods, especially MPE and RRE,
have proved to be superior  convergence accelerators; they require much less computation than, and half as much storage as, the epsilon algorithms for the same accuracy.

In this work, we develop yet another
polynomial type method, which is based on the singular value decomposition (SVD), as well as some ideas that lead to MPE. We denote this new  method by {SVD-MPE}. We also design
a numerically stable  algorithm for its implementation, whose computational cost and storage requirements are minimal.
The new method is described in the next section. In Section \ref{se3}, we show how the error in the approximation produced by SVD-MPE can be estimated at zero cost in terms of the quantities already used in the construction of the approximation. In Section \ref{se4}, we give  a very efficient  algorithm for implementing SVD-MPE. In Section \ref{se5}, we derive determinant representations for the approximations produced by SVD-MPE, while in Section \ref{se6}, we show that this method is a Krylov subspace method when applied to vector sequences that result from the  solution of linear systems via fixed-point iterative schemes.
Finally, in  Section \ref{se7}, we illustrate its use with two numerical examples.

Before closing, we state the ({\em reduced} version of) the well known {\em singular value decomposition} (SVD) theorem. For different  proofs, we refer the reader to  Golub and Van Loan \cite{Golub:2013:MC},
Horn and Johnson \cite{Horn:1985:MA}, Stoer and Bulirsch \cite{Stoer:2002:INA}, and Trefethen and Bau \cite{Trefethen:1997:NLA}, for example.

\begin{theorem}\label{th:SVD} Let $\AAA\in\mathbb{C}^{r\times s}$,
  $r\geq s$. Then there exist unitary matrices $\UU\in\mathbb{C}^{r\times s}$,
  $\VV\in\mathbb{C}^{s\times s}$, and a diagonal matrix $\SSigma=
  \text{\em diag}(\sigma_1,\ldots,\sigma_s)\in\mathbb{R}^{s\times
  s}$, with $\sigma_1\geq\sigma_2\geq\cdots\geq\sigma_s\geq0$,
  such that
  $$ \AAA=\UU\SSigma\VV^*.$$
  Furthermore, if $\UU=[\,\uu_1\,|\,\cdots\,|\,\uu_s\,]$ and
  $\VV=[\,\vv_1\,|\,\cdots\,|\,\vv_s\,]$, then
  $$ \AAA\vv_i=\sigma_i\uu_i, \quad \AAA^*\uu_i=\sigma_i\vv_i,\quad
  i=1,\ldots,s.$$ In case $\text{\em rank}(\AAA)=t$, there holds
  $\sigma_i>0$, $i=1,\ldots,t,$  and the rest of the $\sigma_i$ are zero. \end{theorem}

 \noindent{\bf Remark:} The $\sigma_i$ are called the {\em singular values} of
 $\AAA$ and the $\vv_i$ and $\uu_i$ are called the corresponding {\em right} and {\em left}
 {\em singular vectors} of $\AAA$, respectively. We also have
  $$\AAA^*\AAA\vv_i=\sigma_i^2\vv_i,\quad \AAA\AAA^*\uu_i=\sigma_i^2\uu_i,\quad
  i=1,\ldots,s.$$

\section{Development of SVD-MPE}\label{se2}
\setcounter{theorem}{0} \setcounter{equation}{0}

In what follows, we use boldface lower case letters for vectors and boldface upper case letters for matrices. In addition, we will be working with  general inner products $(\cdot\,,\cdot)$ and the $l_2$ norms $\|\cdot\|$ induced by them: These are defined as follows:
\begin{itemize}
\item In $\C^N$, with $\MM\in\C^{N\times N}$ hermitian positive definite,
\beq \label{eq000} (\aaa,\bb)_{\MM}=\aaa^*\MM\bb,\quad \|\aaa\|_{\MM}=\sqrt{(\aaa,\aaa)_{\MM}}.\eeq
\item In $\C^{k+1}$, $k=1,2,\ldots,$ with $\LL_k\in\C^{(k+1)\times(k+1)}$ hermitian positive definite,
\beq\label{eq001}(\aaa,\bb)_{\LL_k}=\aaa^*\LL_k\bb,\quad \|\aaa\|_{\LL_k}=\sqrt{(\aaa,\aaa)_{\LL_k}}.\eeq
\end{itemize}
Of course, the standard Euclidean  inner product $\aaa^*\bb$  and the $l_2$ norm $\sqrt{\aaa^*\aaa}$ induced by it are obtained by letting $\MM=\II$ in \eqref{eq000} and $\LL_k=\II$ in \eqref{eq001}; we will denote these norms by $\|\cdot\|_2$. (We will denote by $\II$ the identity matrix in every dimension.)

\subsection{Summary of MPE}
We begin with a brief summary of  minimal polynomial extrapolation (MPE). We use the ideas that follow to develop our new method.

Given the vector sequence $\{\xx_m\}$ in $\C^N$, we define
\beq\label{eq1} \uu_m=\xx_{m+1}-\xx_m,\quad m=0,1,\ldots,\eeq and, for some fixed $n$,  define the matrices
$\UU_k$ via
\beq\label{eq2} \UU_k=[\,\uu_n\,|\,\uu_{n+1}\,|\,\cdots\,|\,\uu_{n+k}\,]\in\C^{N\times(k+1)}.\eeq
Clearly, there is an integer $k_0\leq N$, such that  the matrices $\UU_k$, $k=0,1,\ldots, k_0-1$,  are of full rank, but   $\UU_{k_0}$ is not; that is,
\beq \label{eq2a} \rank\,(\UU_k)=k+1,\quad k=0,1,\ldots,k_0-1;\quad \rank\,(\UU_{k_0})=k_0.\eeq
(Of course, this is the same as saying that $\{\uu_0,\uu_1,\ldots,\uu_{k_0-1}\}$ is a linearly independent set,
but $\{\uu_0,\uu_1,\ldots,\uu_{k_0}\}$ is not.)
Following this, we pick a positive integer $k<k_0$ and let the vector $\cc'=[c_0,c_1,\ldots,c_{k-1}]^\text{T}\in \C^{k}$ be the solution to
\beq\label{eq3} \min_{c_0,c_1,\ldots,c_{k-1}}\bigg\|\sum^{k-1}_{i=0}c_i\uu_{n+i}+\uu_{n+k}\bigg\|_{\MM}.\eeq
This minimization problem can also be expressed as in
\beq\label{eq4} \min_{\cc'}\|\UU_{k-1}\cc'+\uu_{n+k}\|_{\MM},\eeq and, as is easily seen, $\cc'$ is the standard least-squares solution to the linear system $\UU_{k-1}\cc'=-\uu_{n+k}$, which, when $k<N$, is overdetermined, and generally inconsistent.
With $c_0,c_1,\ldots,c_{k-1}$ determined, set $c_k=1$, and compute the scalars $\gamma_0,\gamma_1,\ldots,\gamma_k$ via
\beq\label{eq5} \gamma_i=\frac{c_i}{\sum^{k}_{j=0}c_j},\quad i=0,1,\ldots,k,\eeq
provided $\sum^{k}_{j=0}c_j\neq0$. Note that
\beq\label{eq6}\sum^k_{i=0}\gamma_i=1.\eeq
Finally, set
\beq\label{eq7} \sss_{n,k}=\sum^k_{i=0}\gamma_i\xx_{n+i}\eeq
as the approximation to $\sss$, whether $\sss$ is the limit or antilimit of $\{\xx_m\}$.

What we have so far is only the definition (or the theoretical development) of MPE as a method. It should not be taken as an efficient  computational procedure (algorithm), however. For this topic, see
\cite{Sidi:1991:EIM}, where  numerically stable and computationally and storagewise economical algorithms for MPE and RRE are designed for the case in which $\MM=\II$. A well documented FORTRAN 77 code for implementing  MPE and RRE in a unified manner is also provided in \cite[Appendix B]{Sidi:1991:EIM}.

\subsection{Development of SVD-MPE}
We start by observing that the unconstrained minimization problem for MPE given in \eqref{eq4} can also be expressed as a superficially ``constrained'' minimization problem as in
\beq \label{eq4a}\min_{\cc}\|\UU_{k}\cc\|_{\MM},\quad \text{subject to}\quad c_k=1;\quad \cc=[c_0,c_1,\ldots,c_{k}]^\text{T}.\eeq
For the SVD-MPE method,
we replace  this ``constrained'' minimization problem by the following {\em actual} constrained minimization problem:

\beq\label{eq12}\min_{\cc}\|\UU_{k}\cc\|_{\MM},\quad \text{subject to}\quad \|\cc\|_{\LL_k}=1;\quad \cc=[c_0,c_1,\ldots,c_{k}]^\text{T}.\eeq
With $c_0,c_1,\ldots,c_{k}$ determined, we again compute
$\gamma_0,\gamma_1,\ldots,\gamma_k$ via
\beq\label{eq13} \gamma_i=\frac{c_i}{\sum^{k}_{j=0}c_j},\quad i=0,1,\ldots,k,\eeq
provided $\sum^{k}_{j=0}c_j\neq0$, noting again that
\beq\label{eq14}\sum^k_{i=0}\gamma_i=1.\eeq
Finally, we set
\beq\label{eq15} \sss_{n,k}=\sum^k_{i=0}\gamma_i\xx_{n+i}\eeq
as the SVD-MPE approximation to $\sss$, whether $\sss$ is the limit or antilimit of $\{\xx_m\}$.

Of course,  the minimization problem in \eqref{eq12} has a solution for $\cc=[c_0,c_1,\ldots,c_{k}]^\text{T}$. Let   $\sigma_{\min}=\|\UU_k\cc\|_{\MM}$ for this (optimal) $\cc$. Lemma \ref{le11} that follows next gives a complete characterization of $\sigma_{\min}$ and the (optimal) $\cc$.

\begin{lemma} \label{le11}Let $\sigma_{k0},\sigma_{k1},\ldots,\sigma_{kk}$ be the singular values of the $N\times(k+1)$ matrix
\beq \label{eqML1} \widetilde{\UU}_k=\MM^{1/2}\UU_k\LL_k^{-1/2},\eeq
ordered as in
\beq \sigma_{k0}\geq\sigma_{k1}\geq\cdots\geq\sigma_{kk}, \eeq and  let $\hh_{ki}$ be the corresponding right singular vectors of $\widetilde{\UU}_k$, that is,
\beq  \label{eqML2}\widetilde{\UU}_k^*\widetilde{\UU}_k\hh_{ki}=\sigma_{ki}^2\hh_{ki},\quad
\|\hh_{ki}\|_2=1,\quad  i=0,1,\ldots,k.\eeq
Assuming that $\sigma_{kk}$, the smallest singular value of $\widetilde{\UU}_k$, is simple,
the (optimal) solution $\cc$ to the minimization problem in \eqref{eq12} is unique (up to a multiplicative constant $\tau$, $|\tau|=1$), and is given  as in
\beq \label{eqML3}
\cc=\LL_k^{-1/2}{\hh}_{kk};\quad \sigma_{\min}=\|\UU_k\cc\|_{\MM}=\sigma_{kk}.\eeq
\end{lemma}
\begin{proof} The proof is achieved by observing that, with  $\widetilde{\cc}=\LL_k^{1/2}\cc$,
\beq \label{eqML4} \|\UU_{k}\cc\|_{\MM}=\|\widetilde{\UU}_k\widetilde{\cc}\|_2\quad \text{and}\quad
 \|\cc\|_{\LL_k}=\|\widetilde{\cc}\|_2, \eeq
 so that the problem in \eqref{eq12} becomes
 \beq \label{eq12g}
 \min_{\widetilde{\cc}}\|\widetilde{\UU}_k\widetilde{\cc}\|_2,\quad \text{subject to}\quad
 \|\widetilde{\cc}\|_2=1. \eeq
 We leave the details   to the reader.\end{proof}

\noindent{\bf Remarks:} \begin{enumerate}
\item
 In view of the nature of the solution for the  (optimal) $\cc$ involving singular values and vectors, as described in Lemma  \ref{le11}, we call this new method SVD-MPE.
\item
Note that if $\rank\,(\UU_k)=k+1$, then $\rank\,(\widetilde{\UU}_k)=k+1$ too, and we therefore have
$\sigma_{kk}>0$.
\item Of course, $\sss_{n,k}$ exists if and only if the (optimal) $\cc=[c_0,c_1,\ldots,c_k]^\text{T}$ satisfies $\sum^k_{j=0}c_j\neq0$.
    In addition, by \eqref{eq13}, the $\gamma_i$ are unique when $\sigma_{kk}$ is simple.

\end{enumerate}

Before we go on to the development of our algorithm in the next section,  we state the following result concerning  the finite termination property of SVD-MPE, whose proof is very similar to that pertaining to MPE and RRE given in \cite{Sidi:2012:RTV}:

\begin{theorem}\label{thFTP} Let $\sss$ be the  solution to the nonsingular linear system $\xx=\TT\xx+\dd$, and let $\{\xx_m\}$ be the sequence obtained via the fixed-point iterative scheme
$\xx_{m+1}=\TT\xx_m+\dd$, $m=0,1,\ldots,$ with $\xx_0$ chosen arbitrarily. If $k$ is the degree of the minimal polynomial of $\TT$ with respect to $\eepsilon_n=\xx_n-\sss$ (equivalently, with respect to $\uu_n$),\footnote{Given a   matrix $\BB\in\C^{r\times r}$ and a nonzero vector $\aaa\in\C^r,$ the monic polynomial
$P(\lambda)$ is said to be a {\em minimal polynomial of $\BB$ with
respect to $\aaa$} if $P(\BB)\aaa=~\00$ and if $P(\lambda)$ has
smallest degree.\\  It is easy to show that the minimal polynomial $P(\lambda)$  of $\BB$ with respect to
$\aaa$ exists, is unique, and divides the minimal polynomial of
$\BB$, which in turn divides the characteristic polynomial of
$\BB$. $[$Thus, the degree of $P(\lambda)$ is at most $r,$ and its
zeros are some or all of the eigenvalues of $\BB$.$]$ Moreover, if
$Q(\BB)\aaa=\00$ for some polynomial $Q(\lambda)$ with $\deg
Q>\deg P$, then $P(\lambda)$ divides $Q(\lambda)$. Concerning this subject, see Householder \cite{Householder:1964:TMN}, for example.} then $\sss_{n,k}$ produced by  SVD-MPE satisfies $\sss_{n,k}=\sss$.
\end{theorem}

\section{Error estimation}\label{se3}
\setcounter{theorem}{0} \setcounter{equation}{0}
We now turn to the problem of estimating at zero cost the error $\sss_{n,k}-\sss$, whether $\sss$ is the limit or antilimit of $\{\xx_m\}$. Here we assume that $\sss$ is the solution to  the system of  equations
$$ \xx=\ff(\xx);\quad {\ff}:\mathbb{C}^N\to\mathbb{C}^N,\quad \xx\in\mathbb{C}^N,$$
and that the vector sequence $\{\xx_m\}$ is obtained via the fixed-point iterative scheme
$$\xx_{m+1}=\ff(\xx_m),\quad m=0,1,\ldots,$$ $\xx_0$ being the initial approximation to the solution $\sss$.

Now, if $\xx$ is some approximation to $\sss$, then a good measure of the error $\xx-\sss$ in $\xx$ is the {\em residual vector} $\rr(\xx)$ of $\xx$, namely,
$$ \rr(\xx)=\ff(\xx)-\xx. $$ This is justified since $\lim_{\xx\to\sss}\rr(\xx)=\rr(\sss)=\00$.
We consider two cases:
\begin{enumerate}
\item
{\em  $\ff(\xx)$ is linear; that is, $\ff(\xx)=\TT\xx+\dd$, where $\TT\in \C^{N\times N}$ and $\II-\TT$ is nonsingular.}\\
 In this case, we have
 $$\rr(\xx)=\TT\xx+\dd-\xx=(\TT-\II)(\xx-\sss),$$ and, therefore, by  $\sum^k_{i=0}\gamma_i=1$, $\sss_{n,k}=\sum^k_{i=0}\gamma_i\xx_{n+i}$ satisfies
$$\rr(\sss_{n,k})=
 \sum^k_{i=0}\gamma_i[(\TT\xx_{n+i}+\dd)-\xx_{n+i}]=\sum^k_{i=0}\gamma_i(\xx_{n+i+1}-\xx_{n+i})
 = \sum^k_{i=0}\gamma_i\uu_{n+i},$$ and thus
\beq \label{eq91}\rr(\sss_{n,k})
 =\UU_k\ggamma,\quad \ggamma=[\gamma_0,\gamma_1,\ldots,\gamma_k]^\text{T}.\eeq

\item
{\em $\ff(\xx)$ is nonlinear.}\\
In this case, assuming that $\lim_{m\to\infty}\xx_m=\sss$ and expanding $\ff(\xx_m)$ about $\sss$, we have
$$
 \xx_{m+1}=\ff(\sss)+\FF(\sss)(\xx_m-\sss)
 +O(\|\xx_m-\sss\|^2)\quad\text{as $m\to\infty$,}$$
 where $\FF(\xx)$ is the Jacobian matrix of the vector-valued
function $\ff(\xx)$ evaluated at $\xx$.
Recalling that $\sss=\ff(\sss)$, we rewrite this in the
form
 $$ \xx_{m+1}=\sss+\FF(\sss)(\xx_m-\sss)
 +O(\|\xx_m-\sss\|^2)\quad\text{as $m\to\infty$,}$$
from which, we conclude that the vectors $\xx_m$ and
$\xx_{m+1}$ satisfy the approximate equality
 $$\xx_{m+1}\approx
\sss+\FF(\sss)(\xx_m-\sss)\quad
\text{for all large $m$.}$$ That is, for all large $m$,  the
sequence $\{\xx_m\}$ behaves  as if it  were being generated by an
$N$-dimensional approximate {linear} system of the form $(\II-\TT)\xx\approx\dd$
through
$$\xx_{m+1}\approx\TT\xx_m+\dd,\quad m=0,1,\ldots,$$
where $\TT=\FF(\sss)$ and
$\dd=[\II-\FF(\sss)]\sss.$ In view of what we already know about $\rr(\sss_{n,k})$ for linear systems [from \eqref{eq91}], for nonlinear systems, close to convergence, we have
\beq \label{eq92}
\rr(\sss_{n,k})\approx\UU_k\ggamma,\quad \ggamma=[\gamma_0,\gamma_1,\ldots,\gamma_k]^\text{T}.\eeq
 \end{enumerate}

Now, we can compute $\|\UU_k\ggamma\|_{\MM}$ at no cost in terms of the quantities that result from our algorithm, without having to actually compute $ \UU_k\ggamma$ itself. Indeed, we   have the following  result:
\begin{theorem}\label{th:res} Let $\sigma_{kk}$ be the smallest singular value of $\widetilde{\UU}_k$ and let $\hh_{kk}$ be
the corresponding right singular vector. Then the vector $\UU_k\ggamma$ resulting from $\sss_{n,k}$ satisfies
 \beq
\|\UU_k\ggamma\|_{\MM}=\frac{\sigma_{kk}}{\big|\sum^k_{j=0}c_j\big|}.\eeq \end{theorem}
\begin{proof}
First, the solution to \eqref{eq12} is  $\cc=\LL^{-1/2}\hh_{kk}$ by \eqref{eqML3}. Next,
 letting  $\alpha=\sum^k_{j=0}c_j$, we  have $\ggamma=\cc/\alpha$ by \eqref{eq13}.  Consequently,
 $$ \UU_k\ggamma=\frac{\UU_k\cc}{\alpha}. $$ Thus, by Lemma \ref{le11}, we have
 $$\|\UU_k\ggamma\|_{\MM}=\frac{\|\UU_k\cc\|_{\MM}}{|\alpha|}=
 \frac{\sigma_{kk}}{|\alpha|},$$
 which is the required result.\hfill
  \end{proof}

\section{Algorithm for SVD-MPE}\label{se4}
\setcounter{theorem}{0} \setcounter{equation}{0}
We now turn to the design of a good algorithm for constructing numerically the approximation $\sss_{n,k}$ that results from SVD-MPE. We note that matrix computations in floating-point
arithmetic must be done with care, and this is what we would like to achieve here.

In this section, we let $\MM=\II$ and $\LL_k=\II$ for simplicity. Thus, $\widetilde{\UU}_k=\UU_k$. Since there is no room for confusion, we will also use $\sigma_i$,
 $\hh_i$, and $\gg_i$
to denote
$\sigma_{ki}$,  $\hh_{ki}$, and $\gg_{ki}$, respectively.

 As we have seen in Section \ref{se2}, to determine $\sss_{n,k}$, we need $\hh_{k}$, the right singular vector of $\UU_{k}$
 corresponding to its smallest singular value
$\sigma_{k}$.
Now, $\sigma_{k}$ and $\hh_{k}$ can be obtained from the singular value decomposition (SVD) of $\UU_k\in \C^{N\times (k+1)}$.   Of course, the SVD of $\UU_k$ can be computed by applying directly to $\UU_k$ the algorithm of Golub and Kahan \cite{Golub:1965:CSV}, for example. Here we choose to apply SVD to $\UU_k$ in an {\em indirect} way, which will result in a very efficient algorithm  for SVD-MPE  that is economical
both computationally and storagewise in an optimal way.
Here are the details of the computation of the SVD of $\UU_k$, assuming that $\rank\, (\UU_k)=k+1$:
\begin{enumerate}
\item
We first compute the QR factorization of $\UU_k$ in the form
\beq \label{eq31}\UU_k=\QQ_k\RR_k;\quad
 \QQ_k\in \mathbb{C}^{N\times(k+1)},
 \quad
 \RR_k\in \mathbb{C}^{(k+1)\times(k+1)},
 \eeq
 where $\QQ_k$ is unitary (that is, $\QQ_k^*\QQ_k=\II$)  and $\RR_k$ is upper triangular with positive diagonal elements, that is,
 \beq\label{eq32}
  \QQ_k=[\,\qq_0\,|\,\qq_1\,|\,\cdots\,|\,\qq_k\,],
   \quad
\RR_k=\begin{bmatrix}r_{00}&r_{01}&\cdots&r_{0k}\\
&r_{11}&\cdots&r_{1k}\\
&&\ddots&\vdots\\
&&&r_{kk}\end{bmatrix},\eeq
\beq \label{eq33}\qq^*_i\qq_j=\delta_{ij} \quad \forall\ i,j;
 \quad r_{ij}=\qq^*_i\uu_j\quad \forall\ i\leq j;\quad r_{ii}>0\quad \forall\
 i.\eeq
Of course, we can carry out the QR factorizations in different ways. Here we do this by the {\em modified Gram--Schmidt process} (MGS) as follows:
\begin{tabbing}
 m  \=m \=m  \=m \=m  \kill
 1.\> Compute $r_{00}=\|\uu_n\|_2$ and $\qq_0=\uu_n/r_{00}$.\\
 2.\>{\bf For $j=1,\ldots,k$ do} \\
 \>\> Set $\uu^{(0)}_j=\uu_{n+j}$\\
\>\>{\bf For $i=0,1,\ldots,j-1$ do} \\
\>\>\>$r_{ij}=\qq_i^*\uu^{(i)}_j$ and
$\uu^{(i+1)}_j=\uu^{(i)}_j-r_{ij}\qq_i$\\
\>\>{\bf end do $(i)$} \\
\>\>Compute $r_{jj}=\|\uu^{(j)}_j\|_2$ and $\qq_j=\uu^{(j)}_j/r_{jj}$.\\
\>{\bf end do $(j)$}
\end{tabbing}

Note that the matrices $\QQ_{k}$ and $\RR_{k}$ are obtained from
$\QQ_{k-1}$ and $\RR_{k-1}$, respectively, as follows:
\beq\label{eq34} \QQ_{k}=[\,\QQ_{k-1}\,|\,\qq_k\,],\quad
\RR_k=\left[\begin{array}{c|c}\RR_{k-1} &\begin{matrix}r_{0k}\\ \vdots\\
r_{k-1,k}\end{matrix}\\ \hline
\begin{matrix}0&\cdots&0\end{matrix}&
r_{kk}\end{array}\right].\eeq

For MGS, see \cite{Golub:2013:MC}, for example.

\item  We next compute the SVD  of $\RR_k$: By Theorem \ref{th:SVD},  there exist unitary
matrices $\YY,\HH\in\mathbb{C}^{(k+1)\times(k+1)}$,
\beq \label{eq35}\YY=[\,\yy_0\,|\,\yy_1\,|\,\cdots\,|\,\yy_{k}\,],\ \
 \HH=[\,\hh_0\,|\,\hh_1\,|\,\cdots\,|\,\hh_{k}\,];\quad \YY^*\YY=\II,\ \ \HH^*\HH=\II.\eeq
 and a square diagonal matrix $\SSigma\in\mathbb{R}^{(k+1)\times(k+1)}$,
 \beq \label{eq36}\SSigma=\text{diag}\,(\sigma_0,\sigma_1,\ldots,\sigma_{k});\quad
\sigma_0\geq\sigma_1\geq\cdots\geq\sigma_{k}\geq0,\eeq such that
 \beq \label{eq37}\RR_k=\YY\SSigma\HH^*.\eeq
 In addition,  since $\RR_k$ is nonsingular  by our
assumption that $\rank\,(\UU_k)=k+1$, we have that  $\sigma_i>0$ for all $i$. Consequently, $\sigma_{k}>0$.
 \item
 Substituting \eqref{eq37} in \eqref{eq31}, we  obtain the following true singular value decomposition of $\UU_k$:
\beq\label{eq38}\begin{split}\UU_k=\GG\SSigma\HH^*; \quad &\GG=\QQ_k\YY \in\mathbb{C}^{N\times(k+1)} \quad \text{unitary}, \quad\GG^*\GG=\II;\\
&\GG=[\,\gg_0\,|\,\gg_1\,|\,\cdots\,|\,\gg_k\,],\quad \gg_i^*\gg_j=\delta_{ij}.\end{split}\eeq
Thus, $\sigma_i$, the singular values of $\RR_k$,  are also the singular values of $\UU_k$, and
  $\hh_i$, the corresponding  right singular vectors of $\RR_k$, are also the corresponding right singular vectors of $\UU_k$. [Of course, the $\gg_i$ are corresponding left singular vectors of $\UU_k$. Note that, unlike $\YY$, $\HH$, and $\SSigma$, which we {\em must} compute  for our algorithm,  we do {\em not} need to actually compute $\GG$. The mere  knowledge that the SVD of $\UU_k$ is as given in \eqref{eq38}  suffices to conclude that $\cc=\hh_{k}$ is the required optimal solution to \eqref{eq12}, and  continue with the development of our algorithm.]
\end{enumerate}

 With $\cc=\hh_{k}$ already determined, we next compute the $\gamma_i$ as in \eqref{eq13}; that is,
 \beq \ggamma=\frac{\cc}{\sum^k_{j=0}c_j}, \eeq provided $\sum^k_{j=0}c_j\neq0$.

  Next, by the fact that
  $$ \xx_{n+i}=\xx_n+\sum^{i-1}_{j=0}\uu_{n+j}$$ and by
  \eqref{eq14}, we can re-express   $\sss_{n,k}$ in \eqref{eq15} in the form
\beq\label{eq39} \sss_{n,k}=\xx_n+\sum^{k-1}_{j=0}\xi_j\uu_{n+j}=
 \xx_n+\UU_{k-1}\xxi;\quad \xxi=[\xi_0,\xi_1,\ldots,\xi_{k-1}]^\text{T},\eeq
 where the $\xi_i$ are computed from the $\gamma_j$ as in
 \beq \label{eq40}\xi_{-1}=1;\quad\xi_j=\sum^k_{i=j+1}\gamma_i
 =\xi_{j-1}-\gamma_j,\quad
j=0,1,\ldots,k-1. \eeq

Making use of the fact that $\UU_{k-1}=\QQ_{k-1}\RR_{k-1}$, with
\beq \QQ_{k-1}=[\,\qq_0\,|\,\qq_1\,|\,\cdots\,|\,\qq_{k-1}\,],
   \quad
\RR_{k-1}=\begin{bmatrix}r_{00}&r_{01}&\cdots&r_{0,k-1}\\
&r_{11}&\cdots&r_{1,k-1}\\
&&\ddots&\vdots\\
&&&r_{k-1,k-1}\end{bmatrix},\eeq
where the $\qq_i$ and the $r_{ij}$ are exactly those that feature in \eqref{eq32} and \eqref{eq33},
we  next rewrite \eqref{eq39} as in
 \beq \label{eq41}\sss_{n,k}=\xx_n+\QQ_{k-1}(\RR_{k-1}\xxi). \eeq
 Thus, the computation of $\sss_{n,k}$  can be carried out economically as in
 \beq\label{eq43}\sss_{n,k}=\xx_n+\QQ_{k-1}\eeta;\quad
 \eeta=\RR_{k-1}\xxi,\quad\eeta=[\eta_0,\eta_1,\ldots,\eta_{k-1}]^\text{T}.\eeq
Of course, $\QQ_{k-1}\eeta$ is best computed as a linear
  combination of the columns of $\QQ_{k-1}$, hence \eqref{eq43} is computed as in
  \beq \label{eq44}\sss_{n,k}=\xx_n+\sum_{i=0}^{k-1} \eta_i\qq_i.\eeq
It is clear that, for the computation in \eqref{eq43} and \eqref{eq44}, we need to save both $\QQ_k$ and $\RR_k$.

  This completes the design of our algorithm for implementing SVD-MPE. For convenience,  we provide a systematic description of this algorithm in Table \ref{table}.

  Note that the input vectors $\xx_{n+i}$, $i=1,\ldots,k+1,$ need not be saved; actually, they are overwritten by $\uu_{n+i}$, $i=0,1,\ldots,k, $ as the latter are being computed.
  As is clear from the description of MGS given above, we can overwrite the matrix $\UU_k$ simultaneously with the computation of $\QQ_k$ and $\RR_k$,  the vector $\qq_{n+j}$ overwriting $\uu_{n+j}$ as soon as it is computed,  $j=0,1,\ldots,k;$ that is, at any stage of the QR factorization, we store $k+2$ $N$-dimensional vectors in the memory. Since $N>>k$ in our applications, the storage requirement of the $(k+1)\times(k+1)$ matrix  $\RR_k$ is negligible. So is the cost of computing the SVD of $\RR_k$,  and so is the cost of computing the $(k+1)$-dimensional vector  $\eeta$. Thus, for all practical purposes,
    the computational and storage requirements of SVD-MPE are the same as those of MPE.
\medskip

\noindent{\bf Remark:} If we were to compute the SVD of $\UU_k$, namely, $\UU_k=\GG\SSigma\HH^*$, directly, and not by (i)\,first carrying out the QR factorization of $\UU_k$ as $\UU_k=\QQ_k\RR_k$, and (ii)\,then computing the SVD
of $\RR_k$ as $\RR_k=\YY\SSigma\HH^*$, then we would need to waste extra resources  in carrying out the computation of $\sss_{n,k}=\sum^k_{i=0}\gamma_i\xx_{n+i}=\xx_n+\UU_{k-1}\xxi.$
\begin{enumerate}
\item
We would either have to save $\UU_k$  while computing the matrix $\GG$ in its  singular value decomposition. Thus, we would need to save {\em two} $N\times (k+1)$ matrices, namely, $\UU_k$ and $\GG$ in core memory simultaneously.
\item
In case we are worried about storage, therefore,  do   not wish to save $\UU_k$,
we need to recompute the vectors $\xx_n,\xx_{n+1},\ldots,\xx_{n+k}$ in order to compute $\sss_{n,k}$.  Thus, we would need
to compute these vectors {\em twice} to complete the determination of  $\sss_{n,k}$.
\end{enumerate}
Thus, the approach we have proposed here for carrying out the singular value decomposition of $\UU_k$ enables us to save extra computing and storage very conveniently.

\begin{table}[t]
\small
\begin{center}
\fbox{\ \quad
\begin{minipage}{5in}
\begin{itemize}
\item [\textnormal{Step 0.}]
Input: The vectors $\xx_n,\xx_{n+1},\ldots,\xx_{n+k+1}$.
\item [\textnormal{Step 1.}]
Compute \ $\uu_i=\Delta \xx_i=\xx_{i+1}-\xx_i$,\ \ $i=n,n+1,\ldots,n+k$.\\
Set \ $\UU_k=[\,\uu_n\,|\,\uu_{n+1}\,|\,\cdots\,|\,\uu_{n+k}\,]\in\C^{N\times(k+1)}$.\\
Compute the   QR-factorization of $\UU_{k}$, namely,\
$\UU_{k}=\QQ_{k}\RR_{k}$,\\  $\QQ_{k}\in\C^{N\times(k+1)}$,  \ $\RR_{k}\in\C^{(k+1)\times(k+1)}$, \\
$\QQ_{k}$ unitary, $\RR_{k}$ upper triangular with positive diagonal:\\
$ \QQ_{k}=[\,\qq_0\,|\,\qq_1\,|\,\cdots\,|\,\qq_{k}\,],$\ \ \
$ \RR_k=\begin{bmatrix}r_{00}&r_{01}&\cdots&r_{0,k}\\
&r_{11}&\cdots&r_{1,k}\\
&&\ddots&\vdots\\
&&&r_{k,k}\end{bmatrix},$ \\
$\qq_i^*\qq_j=\delta_{ij}\ \ \forall\ i,j; \ \ r_{ij}=\qq_i^*\uu_j\ \ \forall\ i,j, \ \ r_{ii}>0\ \ \forall\ i.$
\item [\textnormal{Step 2.}]
Compute the SVD of $\RR_k$, namely, $\RR_k=\YY\SSigma\HH^*$, \\
$\YY,\HH\in \C^{(k+1)\times(k+1)}$ \ \ unitary,\ \ $\SSigma\in \mathbb{R}^{(k+1)\times(k+1)}$ \ \ diagonal:\\
$\YY=[\,\yy_0\,|\,\yy_1\,|\,\cdots\,|\,\yy_{k}\,]$,\ \
$\HH=[\,\hh_0\,|\,\hh_1\,|\,\cdots\,|\,\hh_{k}\,]$,\ \ $\YY^*\YY=\HH^*\HH=\II,$\\
 $\SSigma=\text{diag}\,(\sigma_0,\sigma_1,\ldots,\sigma_{k}),$ \ \ \
$\sigma_0\geq\sigma_1\geq\cdots\geq\sigma_{k}\geq0$.
\item
[\textnormal{Step 3.}]
Set \ $\cc=[c_0,c_1,\ldots,c_{k}]^\text{T}=\hh_{k}$,\ \ and compute \
$\alpha=\sum^k_{j=0}c_j$.\\
Set \ $\gamma_i=c_i/\alpha$, \ \ $i=0,1,\ldots,k,$\ \ provided $\alpha\neq0$.\\
 Compute \ $\xxi=[\xi_0,\xi_1,\ldots,\xi_{k-1}]^\text{T}$ \  via\\
$\xi_0=1-\gamma_0;\quad \xi_j=\xi_{j-1}-\gamma_j,\quad j=1,\ldots,k-1.$
\item
[\textnormal{Step 4.}]
Compute \ $\sss_{n,k}$\  via \
$\sss_{n,k}=\xx_n+\QQ_{k-1}\RR_{k-1}\xxi$ \ as follows:\\
 Compute \ $\eeta=\RR_{k-1}\xxi$; \
 $\eeta=[\eta_0,\eta_1,\ldots,\eta_{k-1}]^\text{T}$.\\
 Compute \ $\sss_{n,k}=\xx_n+\QQ_{k-1}\eeta=\xx_n+\sum_{i=0}^{k-1} \eta_i\qq_i.$
\end{itemize}
\end{minipage}
}
\caption{\label{table} Algorithm for SVD-MPE.}
\end{center}
\end{table}

\section{Determinant representations for SVD-MPE} \label{se5}
\setcounter{theorem}{0} \setcounter{equation}{0}
In   \cite{Sidi:1986:CSP} and \cite{Sidi:1986:ACV}, determinant representations were derived for the vectors $\sss_{n,k}$ that are produced by the vector extrapolation methods MPE, RRE, MMPE, and TEA. These representations  have turned out to be very useful in the analysis of the algebraic and analytic properties of these methods. In particular, they were used for obtaining interesting  recursion relations among the $\sss_{n,k}$  and in proving sharp convergence and stability theorems for them.  We now derive two analogous determinant representations for $\sss_{n,k}$ produced by SVD-MPE.

The following lemma, whose proof can be found in \cite[Section 3]{Sidi:1986:ACV},
will be used in this derivation in Theorem \ref{th11}.
\begin{lemma}\label{le:det} Let $u_{i,j}$ and $\gamma_j$ be  scalars and let the
$\gamma_j$ satisfy the linear system
 \beq \label{eq:del}\begin{split}&\sum^k_{j=0}u_{i,j}\gamma_j=0,\quad i=0,1,\ldots,k-1,\\
&\sum^k_{j=0}\gamma_j=1.\end{split} \eeq Then, whether $v_j$ are
scalars or vectors, there holds
 \beq\label{eq:det1}
\sum^k_{j=0}\gamma_j\,
v_{j}=\frac{D(v_0,v_1,\ldots,v_k)}{D(1,1,\ldots,1)},\eeq where
\begin{equation}\label{eq:5-2}
D(v_0,v_1,\ldots,v_k)=
\begin{vmatrix}
v_0&v_1&\cdots&v_k\\ u_{0,0}&u_{0,1}&\cdots&u_{0,k}\\
u_{1,0}&u_{1,1}&\cdots&u_{1,k}\\ \vdots&\vdots&&\vdots\\
u_{k-1,0}&u_{k-1,1}&\cdots&u_{k-1,k}\end{vmatrix},
\end{equation} provided $D(1,1,\ldots,1)\neq0$.
In case the $v_i$ are vectors, the determinant
$D(v_0,v_1,\ldots,v_k)$ is defined via its expansion with respect
to its first row.
\end{lemma}

For convenience of notation, we will  write
$$\LL_k=\begin{bmatrix} l_{00}&l_{01}&\cdots&\l_{0k}\\ l_{10}&l_{11}&\cdots&\l_{1k}\\
\vdots&\vdots&& \vdots\\ l_{k0}&l_{k1}&\cdots&\l_{kk}\end{bmatrix},\quad (\LL_k)_{ij}=l_{ij},\quad i,j=0,1,\ldots\ .$$ Then $\LL_{k-1}$ is the principal submatrix of $\LL_k$ obtained by deleting the last row and the last column of $\LL_k$. In addition,
$\LL_{k-1}$ is hermitian positive definite just like $\LL_k$.

Theorem \ref{th11} that follows gives our first determinant representation for $\sss_{n,k}$ resulting from SVD-MPE and is based   only on the smallest singular value $\sigma_{kk}$ of $\widetilde{\UU}_k$ and the corresponding right singular vector $\hh_{kk}$.

\begin{theorem}\label{th11}  Define
\beq \label{eq66}
u_{i,j}=(\uu_{n+i},\uu_{n+j})_{\MM}-\sigma_{kk}^2(\LL_k)_{ij},
\quad i,j=0,1,\ldots,k,\eeq and assume that
\beq \label{eq66m}\sigma_{kk}<\sigma_{k-1,k-1}.\footnote{From the Cauchy interlace theorem, we already know that $\sigma_{kk}\leq\sigma_{k-1,k-1}$.}\eeq
 Then, provided $D(1,1,\ldots,1)\neq0$,    $\sss_{n,k}$   exists and has  the determinant
representation
 \beq\label{eq:determ} \sss_{n,k}=
\frac{D(\xx_n,\xx_{n+1},\ldots,\xx_{n+k})} {D(1,1,\ldots,1)},
\end{equation}
where $D(v_0,v_1,\ldots,v_k)$ is the $(k+1)\times(k+1)$
determinant defined as in \eqref{eq:5-2} in Lemma \ref{le:det}
with the $u_{i,j}$ as in \eqref{eq66}.
\end{theorem}
\begin{proof} With $\widetilde{\UU}_k$ as in \eqref{eqML1},
we start by rewriting \eqref{eqML2} in the form
\beq (\widetilde{\UU}_k^*\widetilde{\UU}_k-\sigma_{kk}^2\II){\hh}_{kk}=\00. \eeq
Invoking here  ${\hh}_{kk}=\LL_k^{1/2}\cc$, which follows from \eqref{eqML3}, and multiplying the resulting equality on the left by $\LL_k^{1/2}$, we obtain
\beq (\UU_k^*\MM\UU_k-\sigma_{kk}^2\LL_k)\cc=\00.\eeq
Dividing both sides of this equality by $\sum^k_{j=0}c_j$,
and invoking \eqref{eq13}, we have
\beq (\UU_k^*\MM\UU_k-\sigma_{kk}^2\LL_k)\ggamma=\00,\eeq
which, by the fact that $$(\UU_k^*\MM\UU_k)_{ij}=\uu_{n+i}^*\MM\uu_{n+j}=(\uu_{n+i},\uu_{n+j})_{\MM},$$ is the same as
\beq \label{eq67}\sum^k_{j=0}[(\uu_{n+i},\uu_{n+j})_{\MM}-\sigma_{kk}^2(\LL_k)_{ij}]\gamma_j=0\quad
\Rightarrow \quad \sum^k_{j=0}u_{i,j}\gamma_j=0,\quad i=0,1,\ldots,k, \eeq
where we have invoked \eqref{eq66}.
We will be able to apply Lemma \ref{le:det} to prove the validity of
 \eqref{eq:determ} if we show that, in \eqref{eq67},  the equations with $i=0,1,\ldots,k-1,$ are linearly independent, or, equivalently, the first $k$ rows of the matrix
 $$\BB_k={\UU}_k^*\MM{\UU}_k-\sigma_{kk}^2\LL_k$$ are linearly independent.
 By the fact that
$$ \UU_k=[\UU_{k-1}\,|\,\uu_{n+k}\,]\quad\text{and}\quad \LL_k=\left[\begin{array}{c|c}
\LL_{k-1}&\ll_k \\ \hline \ll_k^*& l_{kk}\end{array}\right],\quad \ll_k=[l_{0k},l_{1k},\ldots,l_{k-1,k}]^\text{T},$$
we have
$$\BB_k=\left[\begin{array} {c|c} \BB_{k-1}' & \pp \\ \hline \pp^* & \beta\end{array}\right],$$
where
$$\BB_{k-1}'=\UU_{k-1}^*\MM\UU_{k-1}-\sigma_{kk}^2\LL_{k-1},$$ and
$$ \pp=\UU_{k-1}^*\MM\uu_{n+k}-\sigma_{kk}^2\ll_{k},\quad \beta=\uu_{n+k}^*\MM\uu_{n+k}-\sigma_{kk}^2l_{kk}.$$
Invoking $\UU_{k-1}=\MM^{-1/2}\widetilde{\UU}_{k-1}\LL_{k-1}^{1/2}$, we obtain
$$ \BB_{k-1}'=\LL_{k-1}^{1/2}(\widetilde{\UU}_{k-1}^*\widetilde{\UU}_{k-1}-\sigma_{kk}^2\II)
\LL_{k-1}^{1/2}.$$ Since $\sigma_{k-1,k-1}^2$ is the smallest eigenvalue of $\widetilde{\UU}_{k-1}^*\widetilde{\UU}_{k-1}$ and since $\sigma_{k-1,k-1}>\sigma_{kk}$, it turns out that  $\BB_{k-1}'$ is positive definite, which guarantees that the first $k$ rows of $\BB_k$ are linearly independent. This completes the proof.
\end{proof}

\noindent{\bf Remark:} We note that the condition that $D(1,1\ldots,1)\neq0$ in Theorem \ref{th11} is equivalent to the condition  that $\sum^k_{j=0}c_j\neq0$, which we have already met in Section \ref{se2}.
\smallskip

The determinant representation given in Theorem \ref{th19} that follows is based on
the complete singular value decomposition of $\widetilde{\UU}_k$, hence is different from that given in Theorem \ref{th11}. Since there is no room for confusion, we will denote the singular values $\sigma_{ki}$ and  right and left singular vectors   $\hh_{ki}$ and $\gg_{ki}$ of  $\widetilde{\UU}_k$ by $\sigma_i$,  $\hh_{i}$ and $\gg_{i},$  respectively.

\begin{theorem}\label{th19} Let  $\widetilde{\UU}_k$ be as in \eqref{eqML1}, and let $$\widetilde{\UU}_k=\GG\SSigma\HH^*,\quad \GG\in\C^{N\times (k+1)},\quad
\HH\in \C^{(k+1)\times (k+1)},\quad \SSigma\in \mathbb{R}^{(k+1)\times (k+1)}$$
 be the singular value decomposition of $\widetilde{\UU}_k$; that is,
$$ \GG=[\,\gg_{0}\,|\,\gg_{1}\,|\,\cdots\,|\,\gg_{k}\,],\quad \gg_{i}^*\gg_{j}=\delta_{ij};
\quad  \HH=[\,\hh_{0}\,|\,\hh_{1}\,|\,\cdots\,|\,\hh_{k}\,],\quad \hh_{i}^*\hh_{j}=\delta_{ij},$$ and
$$ \SSigma=\text{\em diag}\,(\sigma_{0},\sigma_{1},\ldots,\sigma_{k}),\quad
\sigma_{0}\geq\sigma_{1}\geq\ldots\geq\sigma_{k}.$$
 Define
\beq \label{eq66f}
u_{i,j}=(\MM^{1/2}\gg_{i})^*\uu_{n+j}=\gg_{i}^*\MM^{1/2}\uu_{n+j},
\quad i=0,1,\ldots,k-1,\quad j=0,1,\ldots,k,\eeq
 Then,   $\sss_{n,k}$ has  the determinant
representation
 \beq\label{eq:determ17} \sss_{n,k}=
\frac{D(\xx_n,\xx_{n+1},\ldots,\xx_{n+k})} {D(1,1,\ldots,1)},
\end{equation}
where $D(v_0,v_1,\ldots,v_k)$ is the $(k+1)\times(k+1)$
determinant defined as in \eqref{eq:5-2} in Lemma \ref{le:det}
with the $u_{i,j}$ as in \eqref{eq66f}.
\end{theorem}
\begin{proof}
By Theorem \ref{th:SVD},
   \beq \widetilde{\UU}_k\hh_{k}=\sigma_{k}\gg_{k} \quad\text{and}\quad  \gg_{i}^*\gg_{k}=0,\quad i=0,1,\ldots,k-1\ .\eeq
Therefore,
\beq \label{eqgi1}\gg_{i}^*\widetilde{\UU}_k\hh_{k}=0,\quad i=0,1,\ldots,k-1\ .\eeq
By  \eqref{eqML1} and by the fact that $\cc=\LL_k^{-1/2}\hh_{k}$, which follows from \eqref{eqML3},
and by the fact that $\ggamma=\cc/\alpha$, $\alpha=\sum^k_{j=0}c_j$, which follows from \eqref{eq13},
and by \eqref{eqgi1}, we  then have
\beq  \label{eqgi2}
\gg_i^*\MM^{1/2}\UU_k\ggamma=\alpha^{-1}(\gg_i^*\MM^{1/2}\UU_k\cc)=
\alpha^{-1}(\gg_i^*\widetilde{\UU}_k\hh_k)=0,\quad i=0,1,\ldots,k-1\ .\eeq
But, by \eqref{eq66f},  \eqref{eqgi2} is the same as
$$\sum^k_{j=0}u_{i,j}\gamma_j=0,\quad i=0,1,\ldots,k-1.$$
Therefore, Lemma \ref{le:det} applies with $u_{i,j}$ as in \eqref{eq66f}, and the  result follows.
\end{proof}
  \section{SVD-MPE for linearly generated sequences}\label{se6}
  \setcounter{theorem}{0} \setcounter{equation}{0}
  In Sidi \cite{Sidi:1988:EVP}, we discussed the connection of the extrapolation methods  MPE, RRE, and TEA with Krylov subspace methods. We now want to extend the treatment of
  \cite{Sidi:1988:EVP} to SVD-MPE. Here we recall that a Krylov subspace method is also a
  projection method and that a projection method is defined uniquely by its right and left subspaces.\footnote{A projection method for the solution of the linear system $\CC\xx=\dd$, where $\CC\in \C^{N\times N}$,  is defined as follows: Let ${\cal Y}$ and ${\cal Z}$ be  $k$-dimensional subspaces of $\C^N$ and let $\xx_0$ be a given vector in $\C^N$. Then the projection method produces an approximation $\sss_k$ to the solution of $\CC\xx=\dd$ as follows: (i)\,$\sss_k=\xx_0+\yy$, $\yy\in {\cal Y}$, (ii)\,$\hh^*\rr(\sss_k)=0$ for every $\hh\in{\cal Z}$. ${\cal Y}$ and ${\cal Z}$ are called, respectively,  the right and left subspaces of the method. If ${\cal Y}$ is the Krylov subspace ${\cal K}_k(\CC;\rr_0)=\text{span}\{\rr_0,\CC\rr_0,\ldots,\CC^{k-1}\rr_0\}$, where $\rr_0=\dd-\CC\xx_0$, then the projection method is called a {\em Krylov subspace method.}}
  In the next theorem, we show that SVD-MPE is a bone fide Krylov subspace method and we identify  its right and left subspaces.

  Since there is no room for confusion, we will  use the notation of Theorem \ref{th19}.

  \begin{theorem}\label{th:vs}
  Let $\sss$ be the unique solution to the linear system  $\CC\xx=\dd,$ which we express in the form
  $$ (\II-\TT)\xx=\dd\quad \Rightarrow \quad \xx=\TT\xx+\dd;\quad \TT=\II-\CC,$$ and let the vector sequence $\{\xx_m\}$ be  produced by the fixed-point iterative scheme
  $$ \xx_{m+1}=\TT\xx_m+\dd,\quad m=0,1,\ldots\ .$$ Define the residual vector of $\xx$ via $\rr(\xx)=\dd-\CC\xx.$
  Let also $\sss_k\equiv\sss_{0,k}$ be the  approximation to $\sss$ produced by SVD-MPE. Then
  the following are true:
  \begin{enumerate}
  \item $\sss_k$ is of the form
 \beq\label{eq51}\sss_k=\xx_0+\sum^{k-1}_{i=0}\delta_i(\CC^i\rr_0)\quad \text{for some $\delta_i$};\quad \rr_0=\rr(\xx_0)=\dd-\CC\xx_0.\eeq
   \item \sloppypar
  The residual vector of $\sss_k$, namely, $\rr(\sss_k)$,  is orthogonal to $\text{\em span}\{\MM^{1/2}\gg_0,\MM^{1/2}\gg_1,\ldots,\MM^{1/2}\gg_{k-1}\}$.  Thus,
  \beq\label{eq52}(\MM^{1/2}\gg_i)^*\rr(\sss_k)=0,\quad i=0,1,\ldots,k-1.\eeq
  \end{enumerate}
   Consequently, SVD-MPE is a Krylov subspace method for the linear system $\CC\xx=\dd$,
   with the Krylov subspace ${\cal K}_k(\CC;\rr_0)=\text{\em span}\{\rr_0,\CC\rr_0,\ldots,\CC^{k-1}\rr_0\}$ as its right subspace and
   $\text{\em span}\{\MM^{1/2}\gg_0,\MM^{1/2}\gg_1,\ldots,\MM^{1/2}\gg_{k-1}\}$ as its left subspace.
  \end{theorem}
  \begin{proof}
  With the $\xx_m$ generated as above, we have
  $$\uu_{m+1}=\TT\uu_m \quad\Rightarrow \quad \uu_m=\TT^m\uu_0,\quad m=0,1,\ldots\ . $$
  Therefore, $$\uu_0=\xx_1-\xx_0=\TT\xx_0+\dd-\xx_0=\dd-\CC\xx_0=\rr(\xx_0)$$ and
   $$  \sss_k=\xx_0+\sum^{k-1}_{i=0}\xi_i\uu_i=\xx_0+\sum^{k-1}_{i=0}\xi_i\TT^i\uu_0=
  \xx_0+\sum^{k-1}_{i=0}\xi_i\TT^i\rr_0.$$  Upon substituting $\TT=\II-\CC$ in this equality, we obtain  \eqref{eq51}.

   To prove \eqref{eq52}, we first recall that $\UU_k\ggamma=\rr(\sss_{k})$ by \eqref{eq91}.
  By this and by \eqref{eqgi2}, the result in  \eqref{eq52} follows.
    \end{proof}

\section{Numerical examples}\label{se7}
\setcounter{theorem}{0} \setcounter{equation}{0}
We now provide two examples that show the performance of SVD-MPE and compare it with MPE. In both examples,   SVD-MPE and MPE are   implemented with the standard Euclidean inner product and the norm induced by it. Thus, $\MM=\II$ and $\LL_k=\II$ throughout.

 As we have already mentioned, a major application
area of vector extrapolation methods is that of numerical solution
of large systems of linear or nonlinear equations $\boldsymbol{\psi}(\xx)=\00$ by
fixed-point iterations $\xx_{m+1}=\ff(\xx_m)$. [Here $\xx=\ff(\xx)$ is a possibly preconditioned form of $\ppsi(\xx)=\00$.] For SVD-MPE, as well as all other
polynomial  methods discussed in the literature, the computation of
the approximation $\sss_{n,k}$ to $\sss$, the solution of
$\boldsymbol{\psi}(\xx)=\00$, requires $k+1$ of the vectors $\xx_m$ to be stored
in the computer memory. For systems of very large dimension $N$,
this means that we should keep $k$ at a moderate size.
In view of this limitation, a practical strategy for systems of
equations is {\em cycling}, for which   $n$ and $k$ are fixed. Here
are the steps of cycling:
\begin{enumerate}
\item  [\textnormal{C0.}] Choose integers $n\geq 0$, and  $k\geq
1,$ and an initial vector $\xx_0$.
\item [\textnormal{C1.}]
Compute the vectors $\xx_1,\xx_2,\ldots,\xx_{n+k+1}$ [\,via
$\xx_{m+1}=\ff(\xx_m)$].
\item
[\textnormal{C2.}] Apply SVD-MPE (or MPE)  to the vectors
$\xx_n,\xx_{n+1},\ldots,\xx_{n+k+1}$, with end result $\sss_{n,k}$.
 \item [\textnormal{C3.}]
If $\sss_{n,k}$ satisfies accuracy test, stop.\\
Otherwise, set $\xx_0=\sss_{n,k}$, and go to Step~C1.
\end{enumerate}
We will call each application of steps C1--C3 a {\em cycle}, and
denote by $\sss^{(r)}_{n,k}$ the $\sss_{n,k}$ that is computed in
the $r$th cycle. We will also denote the initial vector $\xx_0$ in
step C0 by $\sss^{(0)}_{n,k}$. Numerical examples suggest that the
sequence $\{\sss^{(r)}_{n,k}\}^\infty_{r=0}$ has very good
convergence properties.

\begin{example}\label{ex1}{\em Consider the vector sequence $\{\xx_m\}$ obtained from $\xx_{m+1}=\TT\xx_m+\dd$, $m=0,1,\ldots,$ where
$$ \TT=0.06\times \begin{bmatrix}
5&2&1&1&&&&&\\ 2&6&3&1&1&&&&\\ 1&3&6&3&1&1&&&\\ 1&1&3&6&3&1&1&&\\
&1&1&3&6&3&1&1&\\ &&\ddots&\ddots&\ddots&\ddots&\ddots&\ddots&\ddots
\end{bmatrix},$$ and is  symmetric with respect to both main diagonals, and $\TT\in\mathbf{R}^{N\times N}$
The vector $\dd$ is such that the exact solution  to $\xx=\TT\xx+\dd$ is $\sss=[1,1,\ldots,1]^\text{T}$. We have $\rho(\TT)<1$, so that $\{\xx_m\}$ converges to $\sss$.

Figure \ref{fig11} shows the $l_2$ norms of the errors in $\sss_{n,k}$, $n=0,1,\ldots,$ with $k=5$ fixed. Here $N=100$.
Note that all of the approximations $\sss_{n,5}$ make use of the same (infinite) vector sequence
$\{\xx_m\}$, and, practically speaking, we are looking at how the methods behave as $n\to\infty$. It is interesting to see that SVD-MPE and MPE behave almost the same. Although we have a  rigorous asymptotic theory confirming the behavior of MPE in this mode as observed in Figure \ref{fig11} (see \cite{Sidi:1986:CSP}, \cite{Sidi:1994:CIR}--\cite{Sidi:1988:CSA}),  we do not have any such theory for SVD-MPE at the present.

Figure \ref{fig12} shows the $l_2$ norm of the error in $\sss_{n,k}$  in the cycling mode with $n=0$ and $k=20$. Now $N=1000$.}
\end{example}

\begin{example}\label{ex2}{\em
We now apply   SVD-MPE and MPE to the nonlinear system that arises from finite-difference approximation of the two-dimensional convection-diffusion equation

$$ -\nabla^2 u+Cu(u_x+u_y)=f, \quad (x,y)\in\Omega=(0,1)\times(0,1),$$ where $u(x,y)$ satisfies  homogeneous boundary conditions.  $f(x,y)$ is constructed  by setting $C=20$ in the differential equation and by taking

$$u(x,y)=10xy(1-x)(1-y)\exp(x^{4.5})$$ as the exact solution.

The equation is discretized on a square grid by approximating $u_{xx}$, $u_{yy}$, $u_x$, and $u_y$ by centered differences with truncation errors $O(h^2)$. Thus, letting $h=1/\nu$, and
$$ (x_i,y_j)=(ih,jh), \quad i,j=0,1,\ldots,\nu,$$ and
$$ u_x(x_i,y_j)\approx \frac{u(x_{i+1},y_j)-u(x_{i-1},y_j)}{2h},\quad
u_y(x_i,y_j)\approx \frac{u(x_{i},y_{j+1})-u(x_{i},y_{j-1})}{2h},$$ and
$$ -\nabla^2u(x_i,y_j)\approx \frac{4u(x_i,y_j)-u(x_{i+1},y_j)-u(x_{i-1},y_j)-u(x_{i},y_{j+1})-u(x_{i},y_{j-1})}{h^2},$$
we replace the differential equation by the finite difference equations
\begin{multline} \frac{4u_{ij}-u_{i+1,j}-u_{i-1,j}-u_{i,j+1}-u_{i,j-1}}{h^2}\notag \\
+Cu_{ij}\frac{u_{i+1,j}-u_{i-1,j}+u_{i,j+1}-u_{i,j-1}}{2h}=f(x_i,y_j),\quad 1\leq i,j\leq \nu-1,\end{multline} with
$$\quad u_{0,j}=u_{i,0}=u_{\nu,j}=u_{i,\nu}=0\quad \forall\ i,j.$$
Here $u_{ij}$ is the approximation to $u(x_i,y_j)$, as usual.

We first write the finite difference equations in a way that is analogous to the PDE
written in the form
$$-\nabla^2 u=f-Cu(u_x+u_y),$$ and split the matrix representing  $-\nabla^2$ to enable the use of the Jacobi and Gauss--Seidel methods as the iterative procedures to generate
the sequences $\{\xx_m\}$.

Figures \ref{fig21} and \ref{fig22} show the $l_2$ norms of the errors in $\sss_{n,k}$ from SVD-MPE and MPE in the cycling mode with
$n=0$ and $k=20$, the iterative procedures being, respectively, that of Jacobi
and that of Gauss--Seidel for the linear part $-\nabla^2 u$ of the PDE. Here $\nu=100$, so that the number of unknowns (the dimension) is $N=99^2$.}
\end{example}

\section*{Acknowledgement}
The author would like to thank Boaz Ophir for carrying out the computations reported in Section \ref{se7} of this work.


\begin{thebibliography}{10}

\bibitem{Brezinski:1970:AEA}
C.~Brezinski.
\newblock Application de l'$\epsilon$-algorithme {\`{a}} la r\'{e}solution des
  syst\`{e}mes non lin\'{e}aires.
\newblock {\em C. R. Acad. Sci. Paris}, 271 A:1174--1177, 1970.

\bibitem{Brezinski:1971:ARS}
C.~Brezinski.
\newblock Sur un algorithme de r\'{e}solution des syst\`{e}mes non
  lin\'{e}aires.
\newblock {\em C. R. Acad. Sci. Paris}, 272 A:145--148, 1971.

\bibitem{Brezinski:1975:GTS}
C.~Brezinski.
\newblock G\'{e}n\'{e}ralisations de la transformation de {Shanks}, de la table
  de {Pad\'{e}}, et de l'$\epsilon$-algorithme.
\newblock {\em Calcolo}, 11:317--360, 1975.

\bibitem{Cabay:1976:PEM}
S.~Cabay and L.W. Jackson.
\newblock A polynomial extrapolation method for finding limits and antilimits
  of vector sequences.
\newblock {\em SIAM J. Numer. Anal.}, 13:734--752, 1976.

\bibitem{Eddy:1979:ELV}
R.P. Eddy.
\newblock Extrapolating to the limit of a vector sequence.
\newblock In P.C.C. Wang, editor, {\em Information {Linkage} {Between}
  {Applied} {Mathematics} and {Industry}}, pages 387--396, New York, 1979.
  Academic Press.

\bibitem{Gekeler:1972:SSE}
E.~Gekeler.
\newblock On the solution of systems of equations by the epsilon algorithm of
  {Wynn}.
\newblock {\em Math. Comp.}, 26:427--436, 1972.

\bibitem{Golub:1965:CSV}
G.H. Golub and W.~Kahan.
\newblock Calculating the singular values and pseudo-inverse of a matrix.
\newblock {\em SIAM J. Numer. Anal.}, Series B, 2:205--224, 1965.

\bibitem{Golub:2013:MC}
G.H. Golub and C.F.~Van Loan.
\newblock {\em {Matrix Computations}}.
\newblock The Johns Hopkins University Press, Baltimore, fourth edition, 2013.

\bibitem{Graves:1983:VVR-1}
P.R. Graves-Morris.
\newblock Vector valued rational interpolants {I}.
\newblock {\em Numer. Math.}, 42:331--348, 1983.

\bibitem{Graves:1992:EMV}
P.R. Graves-Morris.
\newblock Extrapolation methods for vector sequences.
\newblock {\em Numer. Math.}, 61:475--487, 1992.

\bibitem{Horn:1985:MA}
R.A. Horn and C.R. Johnson.
\newblock {\em {Matrix Analysis}}.
\newblock Cambridge University Press, Cambridge, 1985.

\bibitem{Householder:1964:TMN}
A.S. Householder.
\newblock {\em {The Theory of Matrices in Numerical Analysis}}.
\newblock Blaisedell, New York, 1964.

\bibitem{Mesina:1977:CAI}
M.~Me{\u{s}}ina.
\newblock Convergence acceleration for the iterative solution of the equations
  {$X=AX+f$}.
\newblock {\em Comput. Methods Appl. Mech. Engrg.}, 10:165--173, 1977.

\bibitem{Pugachev:1978:ACI}
B.P. Pugachev.
\newblock Acceleration of the convergence of iterative processes and a method
  of solving systems of nonlinear equations.
\newblock {\em U.S.S.R. Comput. Math. Math. Phys.}, 17:199--207, 1978.

\bibitem{Shanks:1955:NTD}
D.~Shanks.
\newblock Nonlinear transformations of divergent and slowly convergent
  sequences.
\newblock {\em J. Math. and Phys.}, 34:1--42, 1955.

\bibitem{Sidi:1986:CSP}
A.~Sidi.
\newblock Convergence and stability properties of minimal polynomial and
  reduced rank extrapolation algorithms.
\newblock {\em SIAM J. Numer. Anal.}, 23:197--209, 1986.
\newblock Originally appeared as NASA TM-83443 (1983).

\bibitem{Sidi:1988:EVP}
A.~Sidi.
\newblock Extrapolation vs. projection methods for linear systems of equations.
\newblock {\em J. Comp. Appl. Math.}, 22:71--88, 1988.

\bibitem{Sidi:1991:EIM}
A.~Sidi.
\newblock Efficient implementation of minimal polynomial and reduced rank
  extrapolation methods.
\newblock {\em J. Comp. Appl. Math.}, 36:305--337, 1991.
\newblock Originally appeared as NASA TM-103240 ICOMP-90-20 (1990).

\bibitem{Sidi:1994:CIR}
A.~Sidi.
\newblock Convergence of intermediate rows of minimal polynomial and reduced
  rank extrapolation tables.
\newblock {\em Numer. Algorithms}, 6:229--244, 1994.

\bibitem{Sidi:2008:VEM}
A.~Sidi.
\newblock Vector extrapolation methods with applications to solution of large
  systems of equations and to {PageRank} computations.
\newblock {\em Comp. \& Maths. with Applics.}, 56:1--24, 2008.

\bibitem{Sidi:2012:RTV}
A.~Sidi.
\newblock Review of two vector extrapolation methods of polynomial type with
  applications to large-scale problems.
\newblock {\em J. Comput. Sci.}, 3:92--101, 2012.

\bibitem{Sidi:1988:CSA}
A.~Sidi and J.~Bridger.
\newblock Convergence and stability analyses for some vector extrapolation
  methods in the presence of defective iteration matrices.
\newblock {\em J. Comp. Appl. Math.}, 22:35--61, 1988.

\bibitem{Sidi:1986:ACV}
A.~Sidi, W.F. Ford, and D.A. Smith.
\newblock Acceleration of convergence of vector sequences.
\newblock {\em SIAM J. Numer. Anal.}, 23:178--196, 1986.
\newblock Originally appeared as NASA TP-2193 (1983).

\bibitem{Sidi:1998:UBC}
A.~Sidi and Y.~Shapira.
\newblock Upper bounds for convergence rates of acceleration methods with
  initial iterations.
\newblock {\em Numer. Algorithms}, 18:113--132, 1998.

\bibitem{Smith:1987:EMV}
D.A. Smith, W.F. Ford, and A.~Sidi.
\newblock Extrapolation methods for vector sequences.
\newblock {\em SIAM Rev.}, 29:199--233, 1987.
\newblock Erratum: {\em SIAM Rev.,} 30:623--624, 1988.

\bibitem{Stoer:2002:INA}
J.~Stoer and R.~Bulirsch.
\newblock {\em {Introduction to Numerical Analysis}}.
\newblock Springer-Verlag, New York, third edition, 2002.

\bibitem{Trefethen:1997:NLA}
L.N. Trefethen and {D. Bau, III}.
\newblock {\em {Numerical Linear Algebra}}.
\newblock SIAM, Philadelphia, 1997.

\bibitem{Wynn:1956:DCT}
P.~Wynn.
\newblock On a device for computing the $e_m({S}_n)$ transformation.
\newblock {\em Mathematical Tables and Other Aids to Computation}, 10:91--96,
  1956.

\bibitem{Wynn:1962:ATI}
P.~Wynn.
\newblock Acceleration techniques for iterated vector and matrix problems.
\newblock {\em Math. Comp.}, 16:301--322, 1962.

\bibitem{Wynn:1963:CFW}
P.~Wynn.
\newblock Continued fractions whose coefficients obey a noncommutative law of
  multiplication.
\newblock {\em Arch. Rat. Mech. Anal.}, 12:273--312, 1963.

\bibitem{Wynn:1964:GPV}
P.~Wynn.
\newblock General purpose vector epsilon algorithm procedures.
\newblock {\em Numer. Math.}, 6:22--36, 1964.

\end{thebibliography}

\newpage

\begin{figure}[p]
\includegraphics{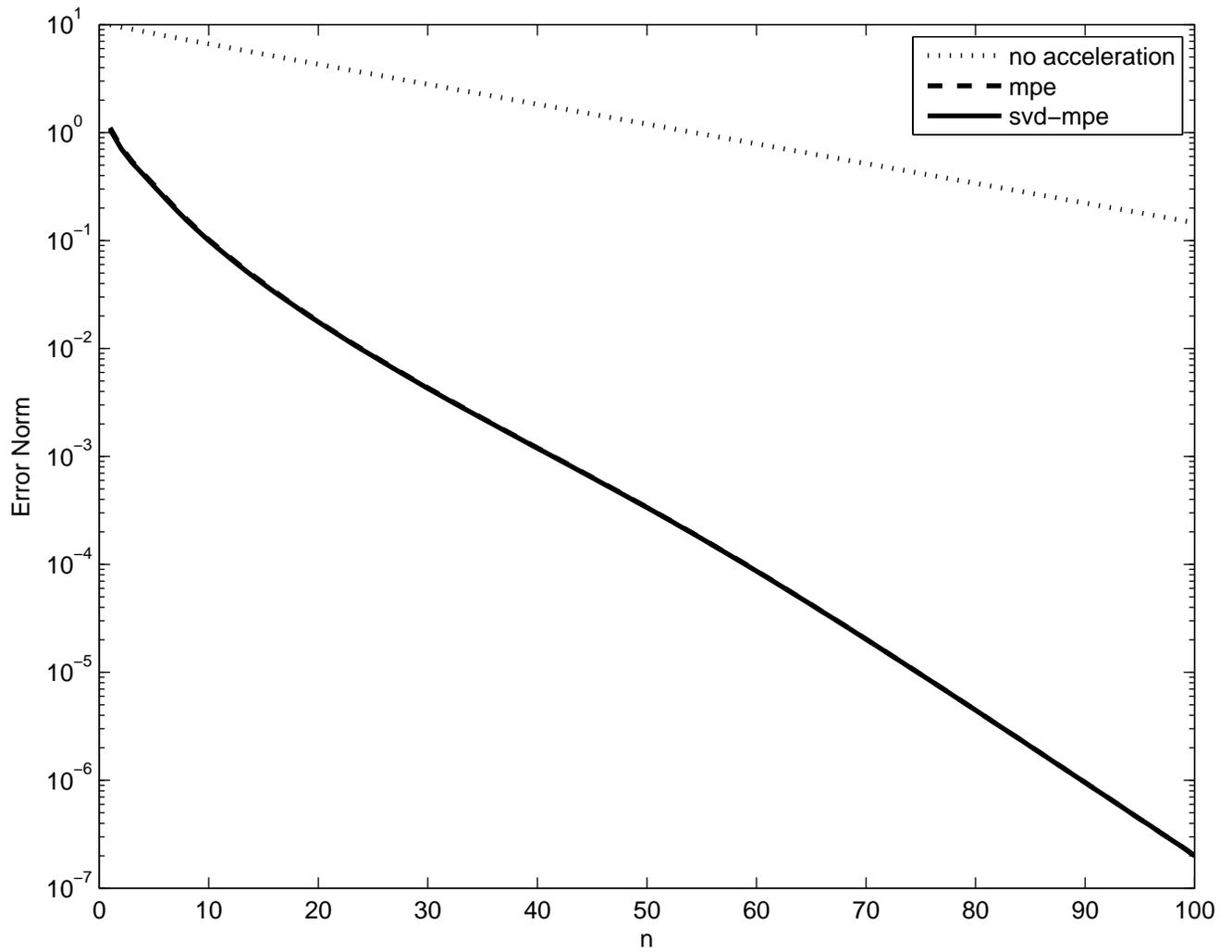}
\caption{\label{fig11} $l_2$ norm of error  in $\sss_{n,k}$, $n=0,1,\ldots,$ with $k=5$,
from MPE and SVD-MPE,
for  Example \ref{ex1} with  $N=100$.}
\end{figure}

\newpage
\begin{figure}[p]
\includegraphics{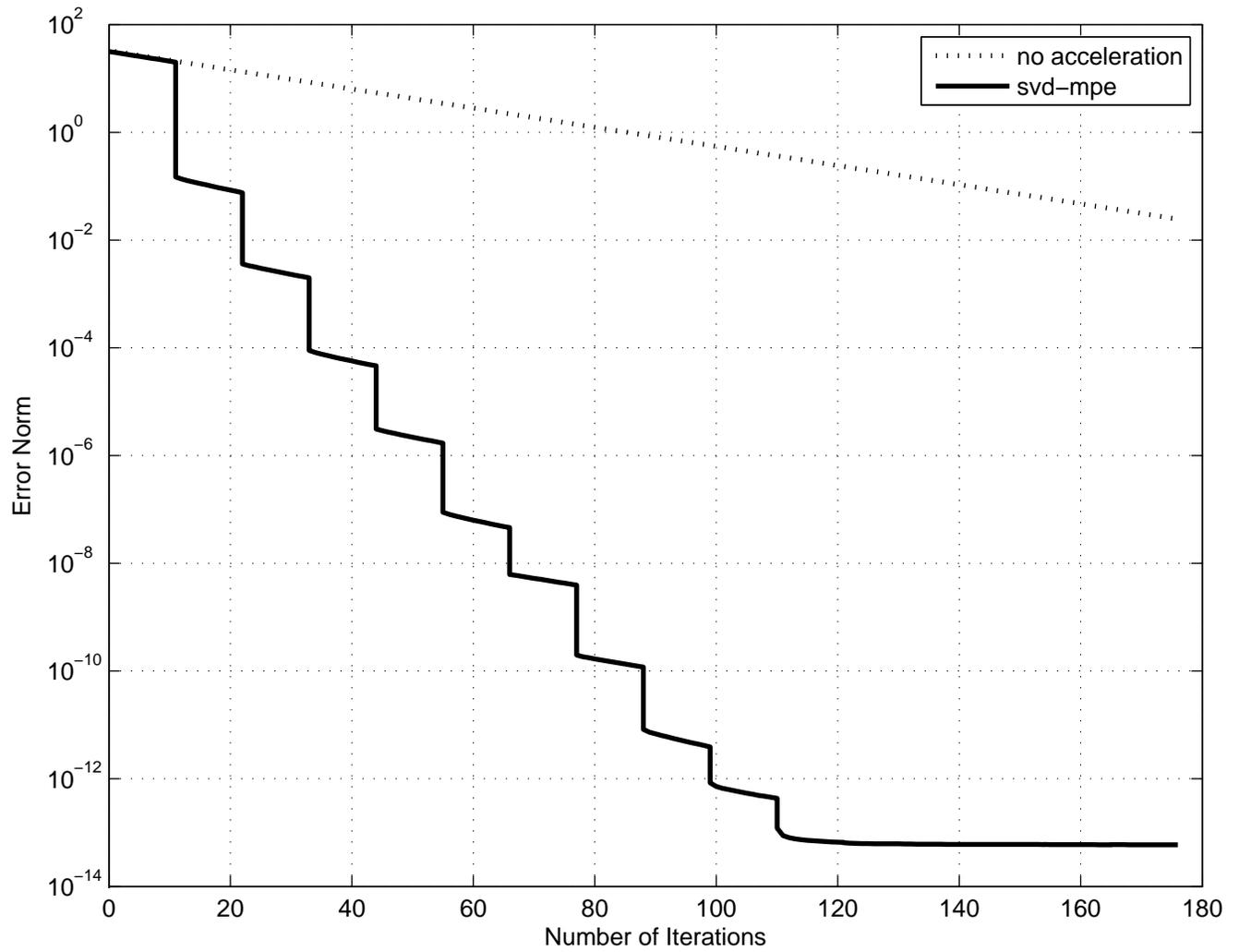}
\caption{\label{fig12} $l_2$ norm of error  in $\sss_{0,20}$ in the cycling mode,
from MPE and SVD-MPE,
for  Example \ref{ex1} with  $N=1000$.}
\end{figure}

\newpage
\begin{figure}[htbp]
\includegraphics{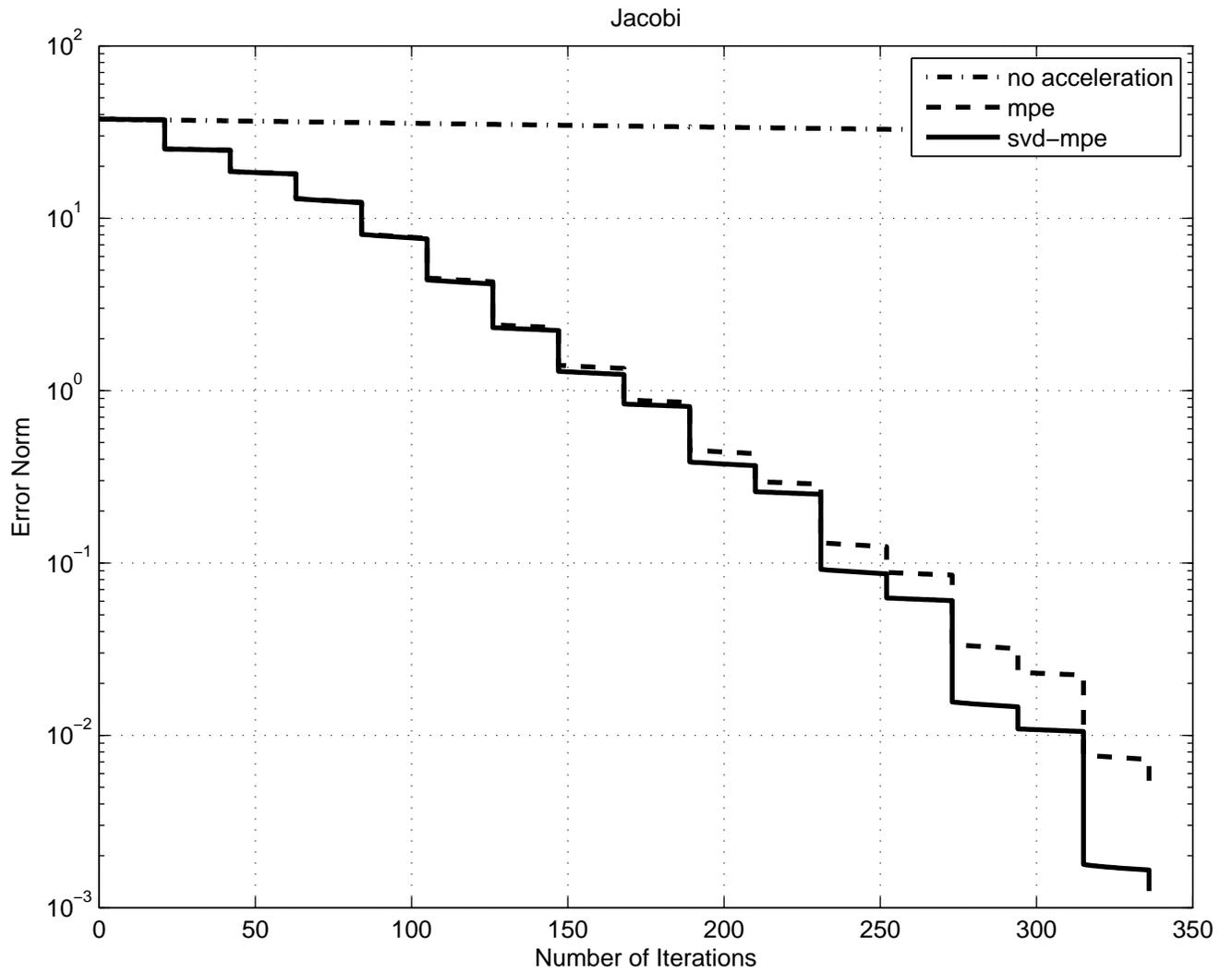}
\caption{\label{fig21} $l_2$ norm of error  in $\sss_{0,20}$ in the cycling mode,
from MPE and SVD-MPE,
for  Example \ref{ex2} with  $\nu=100$ hence $N=99^2$. The underlying  iteration method is that of Jacobi.}
\end{figure}

\newpage
\begin{figure}[htbp]
\includegraphics{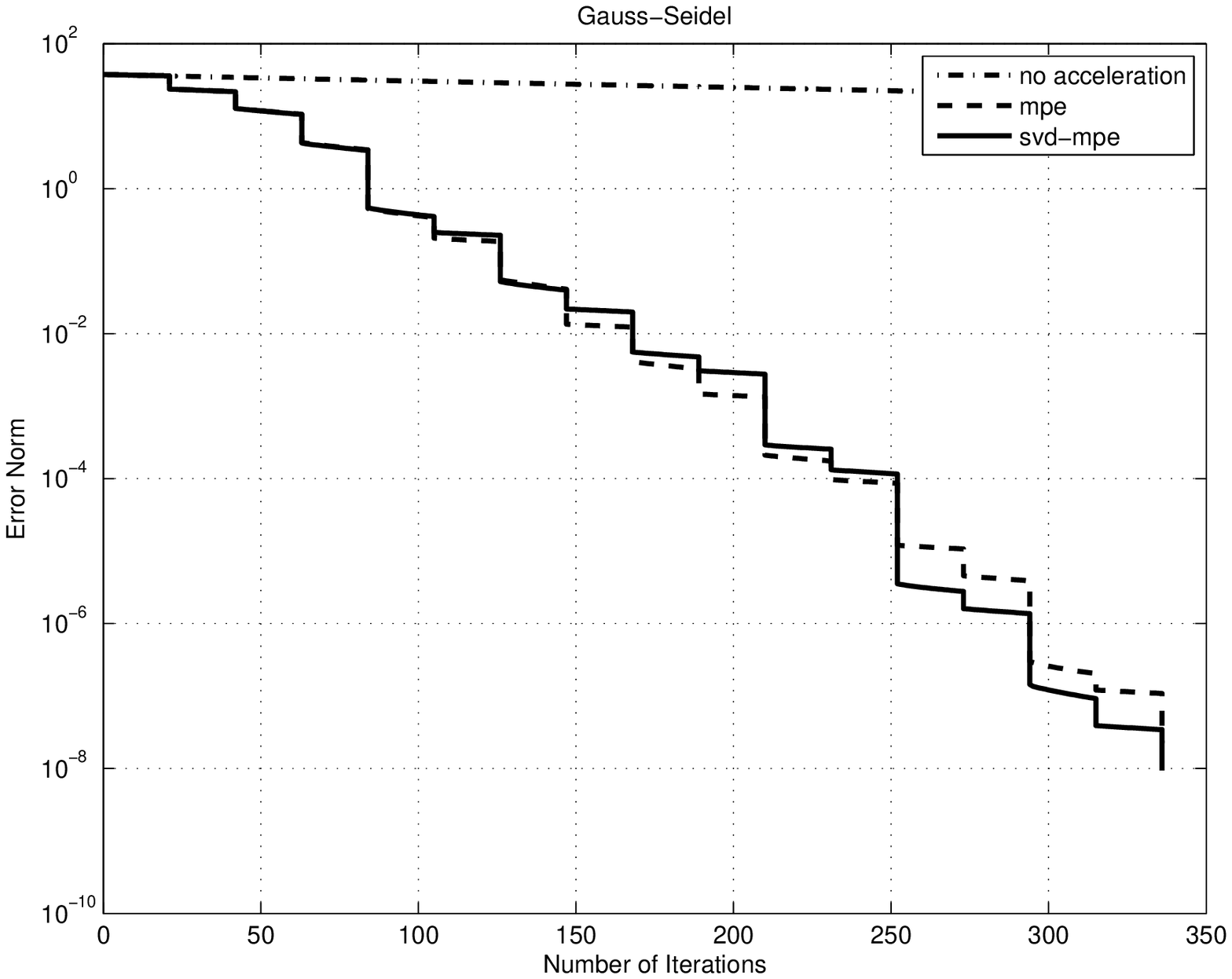}
\caption{\label{fig22} $l_2$ norm of error  in $\sss_{0,20}$ in the cycling mode,
from MPE and SVD-MPE,
for  Example \ref{ex2} with  $\nu=100$ hence $N=99^2$. The underlying  iteration method is that of Gauss--Seidel.}
\end{figure}

\end{document}